\documentclass[11pt,letter]{article} 

\usepackage{indentfirst}
\usepackage{graphicx}
\usepackage{amsfonts,amsmath,amssymb,amsthm}
\usepackage{booktabs,subcaption,amsfonts,dcolumn}
\usepackage{hyperref}
\usepackage[square,sort,comma,numbers]{natbib}
\usepackage[toc,page]{appendix}
\usepackage{appendix}
\usepackage{url}
\DeclareMathOperator*{\argmax}{argmax}

\newtheorem{thm}{Theorem}[section]
\newtheorem{prop}[thm]{Proposition}

\newtheorem{lemA}{Lemma}[section] 							
\newtheorem{cor}[thm]{Corollary}

\newtheorem*{asm*}{Assumptions}

\theoremstyle{remark}

\newtheorem*{rem*}{Remark}

\theoremstyle{definition}

\title{Best Strategy for Each Team in The Regular Season 
to Win Champion in The Knockout Tournament}

\author{Zijie Zhou\\
\\
\normalsize{Department of Mathematics, Purdue University,}\\
\normalsize{zhou759@purdue.edu}
\\
\\
Mentor: Jonathon Peterson\\
\\
\normalsize{Professor of Department of Mathematics, Purdue University,}\\
\normalsize{peterson@purdue.edu}
}

\begin{document}

\maketitle

\begin{abstract}
In `J. Schwenk.(2018) \citep{schwenk2000correct} What is the Correct Way to Seed a Knockout Tournament? Retrieved from The American Mathematical Monthly' , Schwenk identified a surprising weakness in the standard method of seeding a single elimination (or knockout) tournament. In particular, he showed that for a certain probability model for the outcomes of games it can be the case that the top seeded team would be less likely to win the tournament than the second seeded team. 
This raises the possibility that in certain situations it might be advantageous for a team to intentionally lose a game in an attempt to get a more optimal (though possibly lower) seed in the tournament. 
We examine this question in the context of a four team league which consists of a round robin ``regular season'' followed by a single elimination tournament with seedings determined by the results from the regular season \citep{vu2011fair}. Using the same probability model as Schwenk we show that there are situations where it is indeed optimal for a team to intentionally lose. Moreover, we show how a team can make the decision as to whether or not it should intentionally lose. We did two detailed analysis. One is for the situation where other teams always try to win every game. The other is for the situation where other teams are smart enough, namely they can also lose some games intentionally if necessary. The analysis involves computations in both probability and (multi-player) game theory. 
\end{abstract}

\section{Introduction}
In contemporary society, sport competitions such as NBA, NCAA basketball, baseball are more and more prevalent and attracting. In most of these competitions, every team in the knockout tournament has to play head-to-head matches to eliminate the rival and finally tries best to win the tournament. Whether the knockout tournament is fair and what strategy each team has under the knockout tournament is sparking argue between fans every day. In this article, we use the single elimination tournament model created by J. Schwenk.(2018) \citep{schwenk2000correct}. We assume that there are four teams in the playoff: $a_1$, $a_2$, $a_3$, and $a_4$. Each of them has a weight, $v_1$, $v_2$, $v_3$, $v_4$, respectively, which shows the strength of a team. The larger weight, the stronger the team is. Suppose that $v_1 \geq v_2 \geq v_3 \geq v_4$, here $a_1$ is the best team in the knockout tournament and we will give the best strategy for it. Let the probability team $v_i$ beats $v_j$ be $\frac{v_i}{v_i+v_j}$. The schedule of the knockout tournament is in figure \ref{fig:table}. In the first round, the first seed plays against the fourth seed. The second seed plays against the third seed.  In the second round, the winner between the first seed and the fourth seed plays against the winner between the second seed and the third seed. According to the model created by J. Schwenk \citep{schwenk2000correct}, he showed that under some specific situation it is possible for a lower seeded team to enter the tournament with a larger probability to win the tournament than the highest seeded team.  This implies that a team might intentionally lose a game during the regular season to become a lower seed to enter the tournament. If all other teams try to become higher-seeded team, it is not hard for us to make a strategy which we can enter the tournament with the proper seed which has the highest probability to win the champion. However, it is more interesting that if all other teams also lose some games intentionally to maximize the probability to win the tournament. In this situation, we are faced with a game theory problem. We will show how to find the best strategy in this situation and whether every team can maximize the probability to win the tournament respectively.

\begin{figure}[h]
\begin{center}
\includegraphics[width=2.5in,height=2.5in]{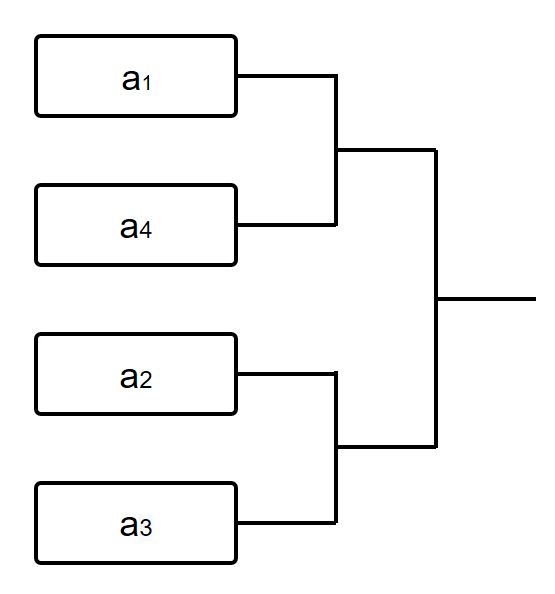}
\caption{The schedule of the tournament model}\label{fig:table}
\end{center}
\end{figure}

\subsection{Main Contributions}
We build a regular season model with four teams and conduct a detailed analysis of the best team's strategy, which can help the best team to decide whether to win or lose intentionally in every week in the regular season. It is noteworthy that this model is based on assumption that other teams try their best to win every match, or we can say other teams are not smart enough. We called this Four Teams Regular Season with Not Smart Enough Rivals Model (FRNS). However, in the real world, every professional team is smart enough. Every team has intelligent people to make decisions for them. Thus, we build another model, which assumes that all teams are smart enough and are able to make the most correct decision for them. We call this Four Teams Regular Season with Smart Enough Rivals Model (FRS). More importantly, FRS can get strategy for every team in the regular season. 

In game theory, we define pure strategy as a strategy which determines the move a player will make for any situation they could face. We define mixed strategy as an assignment of probability to each pure strategy.
In the FRNS model, undoubtedly, the strategy we give to the best team in each week depends on other teams' weights and the current performance of each team. We define the $\pi_i$ as the action variable for team $a_i$. That is $\pi_i=\alpha$ represents team $a_i$ tries to win with probability $\alpha$ and loses intentionally with probability $1-\alpha$, where $\alpha \in [0,1]$. We define $\Pi$ as the collection of vectors $\pi=(\pi_1,\pi_2,\pi_3,\pi_4)$, i.e. $\Pi = [0,1]^4$. We also define $P_{\pi}^{(i)}(m,W,V)$ as the probability for team $a_i$ to win the champion under action vector $\pi \in \Pi$, team weight vector $V$ and performance vector $W$ in week $m$. The team weight vector $V=[v_1, v_2, v_3, v_4]$ which contains team weights for all team. The performance vector $W=[W_1,W_2,W_3,W_4]$, where $W_i$ represents the number of winning games before week $m$ for team $a_i$. Thus, our objective is to find $\sup_{\pi = (\alpha,1,1,1)}P_{\pi}^{(1)}(m,W,V)$ for every $m$, $W$ and $V$. We claim that: under the FRNS model, for all $m$, $W$ and possible $V$, $\argmax_{\pi = (\alpha,1,1,1)}P_{\pi}^{(1)}(m,W,V) \in \{(0,1,1,1),(1,1,1,1)\}$, which shows that the best strategy for the best team in every week is pure strategy. Intuitively, the reason is that other teams' strategies are fixed, i.e. $\pi_2, \pi_3, \pi_4 = 1$.

However, in the FRS model, the action of every team is not fixed due to all teams are smart enough. Our objective is to find $\sup_{\pi \in \Pi}P_{\pi}^{(i)}(m,W,V)$ for every $m$, $W$, $V$ and $ i \in \{1,2,3,4\}$. To our surprise, the final result shows that the best strategy of every team is not mixed strategy, but pure strategy for almost all possible team weight. Every team can have mixed strategy only in a special case where team $a_1$ is as strong as $a_2$, at the same time team $a_3$ is as strong as $a_4$.
\begin{thm} \label{frs}
Under the FRS model, for all $m$, $W$, $\pi=(\pi_1,\pi_2,\pi_3,\pi_4)$ satisfies that for each $i \in \{1,2,3,4\}$, $P_\pi^{(i)}(m,W,V) \geq P_{\pi^{'}}^{(i)}(m,W,V)$ for $\forall \pi^{'}$ s.t. $\pi_j^{'}=\pi_j$, $j \neq i$ is a mixed strategy, i.e. $\pi \in (0,1)^4$ if and only if $v_1=v_2$, $v_3=v_4$. Otherwise, $\pi$ is a pure strategy, i.e. $\pi \in \{0,1\}^4$.
\end{thm}

\section{Analysis for the FRNS Model}
\subsection{Description and Assumption of FRNS Model}
In this section, we want to analyze the `regular season' to get a strategy for the best team to decide whether to try to win or lose intentionally in each 
game. Our logic is that first to analyze the last game in the `regular season', second to analyze the last two games, third to analyze the last three games, and so on. Before doing the analysis, we will first introduce our FRNS model.

The FRNS model is that we have four teams, $a_1$, $a_2$, $a_3$, $a_4$. Their weight is $V=[v_1,v_2,v_3,v_4]$. Assume that $v_1 \geq v_2 \geq v_3 \geq v_4$, so we will help team $a_1$ to get a strategy $\pi_1 \in [0,1]$. The winning probability in a single game between $a_i$ and $a_j$ is $p_{ij}=\frac{v_i}{v_i+v_j }$ \citep{schwenk2000correct}. We suppose that except $a_1$, other teams will try their best to win for every match, i.e. $\pi_2=\pi_3=\pi_4=1$. Define the subset $(\alpha,1,1,1)=\hat{\Pi} \subseteq \Pi$, where $\alpha \in [0,1]$, then $\pi_1$ is such that $(\pi_1,1,1,1)=\argmax_{\pi \in \hat{\Pi}}P_{\pi}^{(1)}(m,W,V)$ such that for any week $m$, performance vector $W$, and possible $V$. In addition, we build a particular schedule for last three weeks (see figure \ref{fig:sch1}).

\begin{figure}[h]
\begin{center}
\includegraphics[width=4in,height=1.5in]{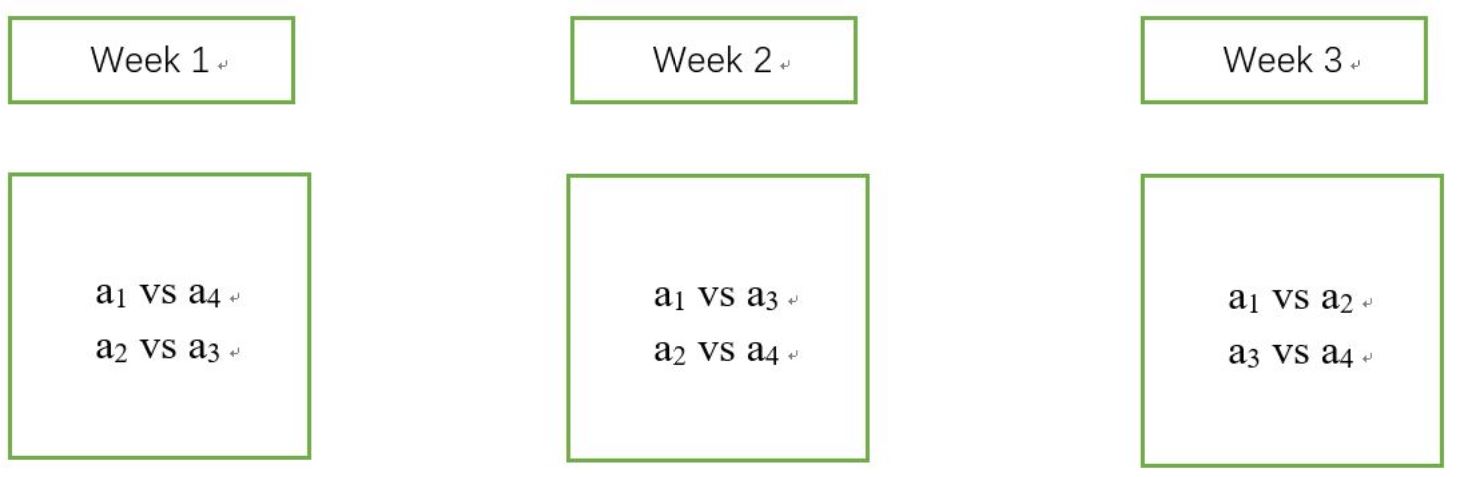}
\caption{A particular schedule for last three weeks}\label{fig:sch1}
\end{center}
\end{figure}

Since there are four teams, there are $4!=24$ seedings in the knockout tournament. However, there exists three types of knockout tournaments, we ignore the exact rank of each team and we only care that which two teams will have a battle in the first week (see figure \ref{fig:sch2}). Here we need another assumption:
\begin{asm*}
If the number of winning games of two teams is the same, then they will flip a fair coin to decide their seeding in the knockout tournament. 
\end{asm*}

For example, if finally $a_1$ and $a_2$ win $2$ games, $a_3$ and $a_4$ win 1 game, then in this situation, $a_1$ and $a_2$ will have same probability, $50\%$, to be the first and the second seed. $a_3$ and $a_4$ will have same probability, $50\%$, to be the third and the fourth seed. If finally $a_1$ wins $3$ games and all of $a_2$, $a_3$ and $a_4$ win 1 game, then, $a_1$ is the first seed and $a_2$, $a_3$ and $a_4$ will have same probability, $33.33\%$, to be the second, third, and the fourth seed.

\begin{figure}[h]
\includegraphics[width=2in,height=3in]{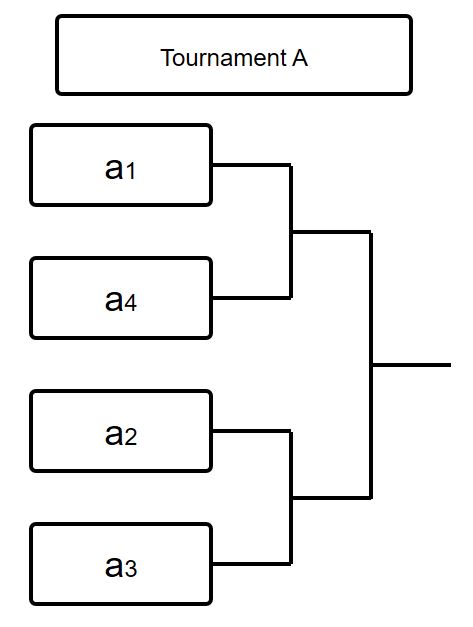}
\includegraphics[width=2in,height=3in]{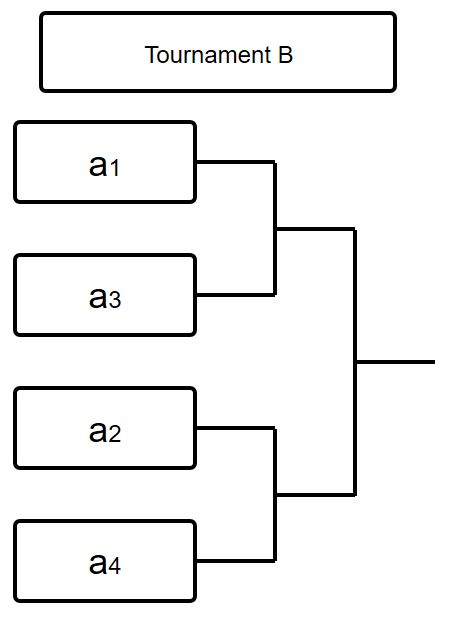}
\includegraphics[width=2in,height=3in]{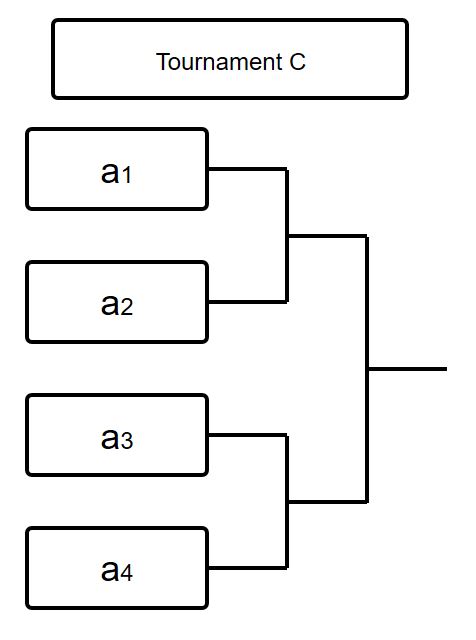}
\caption{three kind of knockout tournaments}\label{fig:sch2}
\end{figure}

An interesting idea is that we find the probability for $a_1$ to win the champion in tournament $A$ is always larger than the one in tournament B and tournament C, no matter what vector $V$ is. Suppose that the probability for $a_1$ to win the champion in tournament $A$, $B$, $C$ is $T_A$, $T_B$, $T_C$, respectively. 

\begin{thm}\label{thm:ordertourn}
For any weight vector $V$, subject to $v_1 \geq v_2 \geq v_3 \geq v_4$, we have
\begin{equation}
T_A \geq T_B \geq T_C
\end{equation}
\end{thm}

The proof of theorem \ref{thm:ordertourn} can be found in appendix $A.1$. 

We have introduced that our logic is to first analyze the strategy in last week. In the last week, there exists fifteen different $W$ vectors: $[2,2,0,0]$, $[2,1,0,1]$, $[2,1,1,0]$, $[2,0,1,1]$, $[1,1,1,1]$, $[1,2,0,1]$, $[1,2,1,0]$, $[1,1,2,0]$, $[1,0,2,1]$, $[1,1,0,2]$, $[1,0,1,2]$, $[0,1,1,2]$, $[0,0,2,2]$, $[0,1,2,1]$, $[0,2,1,1]$. In the next part, we will pick one specific $W$ vector to do analysis as an example.

\subsection{Example $W=[2,2,0,0]$}
If at the beginning of last week, the performance vector $W$ is $[2,2,0,0]$, we first consider the game $a_3$ vs $a_4$. There are two possible results: $a_3$ wins or $a_4$ wins. If $a_3$ wins, $W$ will become $[2,2,1,0]$. If $a_4$ wins, $W$ will become $[2,2,0,1]$. Then we will analyze the match $a_1$ vs $a_2$. If $a_1$ tries to win, two situations may happen: $a_1$ wins and a$_1$ loses. However, if $a_1$ wants to lose intentionally, the only possible result is $a_1$ loses. Assume that $a_3$ defeats $a_4$ and $W=[2,2,1,0]$, if $a_1$ defeats $a_2$, $W$ will become $[3,2,1,0]$, which leads to tournament $A$. If $a_2$ defeats $a_1$, $W$ will become $[2,3,1,0]$, which leads to tournament $B$. 

Now we can calculate the probability for $a_1$ to win the tournament based on different strategy $\pi_1$ in last match. Recall that $a_3$ defeats $a_4$ with probability $p_{34}$, and if it happens, $W$ will become $[2,2,1,0]$. $a_4$ defeats $a_3$ with probability $p_{43}$, and if it happens, $W$ will become $[2,2,0,1]$.

\begin{equation}
W=
\left\{ 
	\begin{array}{lr}
	$[2,2,1,0]$, & \text{with probability $p_{34}$} \\
	$[2,2,0,1]$, & \text{with probability $p_{43}$}	
	\end{array}
\right.
\end{equation} 

If $a_1$ tries to win the last match, the analysis is in table \ref{tab:table2}. Thus, the total probability for $a_1$ to win the champion if $a_1$ tries to win in the last match is $P_{\pi=(1,1,1,1)}^{(1)}(3,[2,2,0,0],V)=p_{12}p_{34}T_A+p_{12}p_{43}T_B+p_{21}p_{34}T_B+p_{21}p_{43}T_A$. If $a_1$ loses intentionally, then the analysis is in table \ref{tab:table3}. Thus, the total probability for $a_1$ to win the champion if $a_1$ loses intentionally in the last match is $P_{\pi=(0,1,1,1)}^{(1)}(3,[2,2,0,0],V)=p_{34}T_B+p_{43}T_A$. Under particular weight vector $V=[v_1,v_2,v_3,v_4]$, we compare $P_{\pi=(1,1,1,1)}^{(1)}(3,[2,2,0,0],V)$ and $P_{\pi=(0,1,1,1)}^{(1)}(3,[2,2,0,0],V)$, the larger one represents the strategy for $a_1$. 

\newcolumntype{d}[1]{D..{#1}}
\newcommand\mc[1]{\multicolumn{1}{c}{#1}}
\begin{table}
\begin{tabular}{| m{0.5in} | m{0.5in}| m{1in}| m{1in}| m{2in} |}
\toprule
 a$_1$ &  a$_3$ & Win Vector (W) & Tournament & Probability for team $a_1$ to win champion\\
  
\midrule
Win & Win & $[3,2,1,0]$ & A & $p_{12}p_{34}T_A$\\
\midrule
Win & Lose & $[3,2,0,1]$ & B & $p_{12}p_{43}T_B$\\
\midrule
Lose & Win & $[2,3,1,0]$ & B & $p_{21}p_{34}T_B$\\
\midrule
Lose & Lose & $[2,3,0,1]$ & A & $p_{21}p_{43}T_A$\\   
\bottomrule
\end{tabular}
\caption{analysis table if $a_1$ tries to win}
\label{tab:table2}

\end{table} 

\begin{table}
\begin{tabular}{| m{0.5in} | m{0.5in}| m{1in}| m{1in}| m{2in} |}
\toprule
 a$_1$ &  a$_3$ & Win Vector (W) & Tournament & Probability for team $a_1$ to win champion\\
  
\midrule
Lose & Win & $[2,3,1,0]$ & B & $p_{34}T_B$\\
\midrule
Lose & Lose & $[2,3,0,1]$ & A & $p_{43}T_A$\\ 
\bottomrule
\end{tabular}
\caption{analysis table if $a_1$ loses intentionally}
\label{tab:table3}

\end{table}   

\begin{thm} \label{alwayswin}
Under $W=[2,2,0,0]$ at the beginning of the last week, $a_1$ should always try to win no matter what weight vector $V$ is.
\end{thm}

The proof of theorem \ref{alwayswin} can be found in appendix $A.2$. Next, we will show the analytical results for all $W$.

\subsection{Analytical Results}

We assume that $v_1 = 1$, $v_2$, $v_3$, $v_4$ are in $(0,1)$, and $v_2 \geq v_3 \geq v_4$. We let $v_2$ be the $x$ axis, $v_3$ be the $y$ axis, $v_4$ be the $z$ axis. The region in the plots is the set $\{V=(1,v_2,v_3,v_4) | P_{\pi=(0,1,1,1)}^{(1)}(3,W,V)-P_{\pi=(1,1,1,1)}^{(1)}(3,W,V) \geq 0, v_2 \geq v_3 \geq v_4\}$ for given $W$. Here are the results.

For $W=[2,2,0,0]$, $[2,1,1,0]$, $[1,1,2,0]$, $[1,1,1,1]$, $[1,0,2,1]$, $[0,2,1,1]$, $[0,1,2,1]$, $[0,0,2,2]$, since we have completed the analysis of the situation where $W=[2,2,0,0]$ by theorem \ref{alwayswin}, the analysis of rest seven situations of $W$ is similar to the $W=[2,2,0,0]$ example. After some calculation, we can know that for every $V$, $P_{\pi=(0,1,1,1)}^{(1)}(3,W,V)-P_{\pi=(1,1,1,1)}^{(1)}(3,W,V) \leq 0$. Thus, it is always sensible for $a_1$ to try to win the last match if these eight situations happen.

For $W=[2,1,0,1]$, $[2,0,1,1]$, $[1,1,0,2]$, $[1,0,1,2]$, by the similar method, we find that $P_{\pi=(0,1,1,1)}^{(1)}(3,W,V)-P_{\pi=(1,1,1,1)}^{(1)}(3,W,V) \geq 0$ holds for every $V$. We can conclude that the region in figure \ref{fig:f6} is the whole region, i.e. $\{V=(1,v_2,v_3,v_4) | v_2 \geq v_3 \geq v_4\}$. Thus, it is always sensible for $a_1$ to lose the last match intentionally if these four situations happen.

\begin{figure}[h]
\begin{center}
\includegraphics[width=1.5in,height=1.5in]{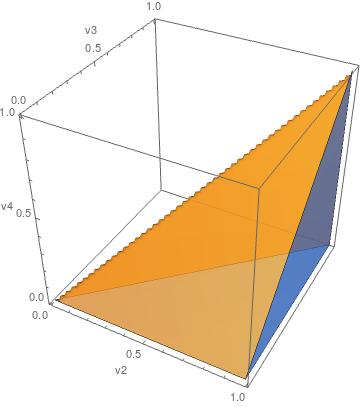}
\caption{the 3d region plot for the difference of two values in $[2,1,0,1]$, $[2,0,1,1]$, $[1,1,0,2]$, $[1,0,1,2]$}\label{fig:f6}
\end{center}
\end{figure}

For $[1,2,0,1]$ and $[0,1,1,2]$, similarly, we can find that the strategy for $a_1$ to decide whether to try to win or lose intentionally in the last match depends on the team weight vector $V=(1,v_2,v_3,v_4)$. Different team weight leads to different strategy. 

\begin{thm} \label{1201t}
Under a specific team weight vector $V=(1,v_2,v_3,v_4)$, $W=[1,2,0,1]$ or $[0,1,1,2]$, then $a_1$ should try to win if and only if the following inequality holds. Otherwise, $a_1$ should lose intentionally.
\begin{align*} \label{1201}
& 3v_2v_3v_4^2+2v_2v_4^3+6v_2^2v_3^2v_4+v_2v_3^2v_4+v_3^3v_4+2v_3v_4^3+10v_3^2v_4^2-3v_2^2v_3^3 \\ & -
v_2^2v_3v_4^2-2v_2^2v_4^3-v_2v_3^4-3v_2v_3^3v_4-2v_2v_3^2v_4^2-v_3^4v_4-6v_3^3v_4^2-6v_3^2v_4^3 \leq 0
\end{align*}
\end{thm}

The proof of theorem \ref{1201t} can be found in appendix $A.3$. The plot of this region is figure \ref{fig:f7}, which is made by Mathematica. 

\begin{figure}[h]
\begin{center}
\includegraphics[width=1.5in,height=1.5in]{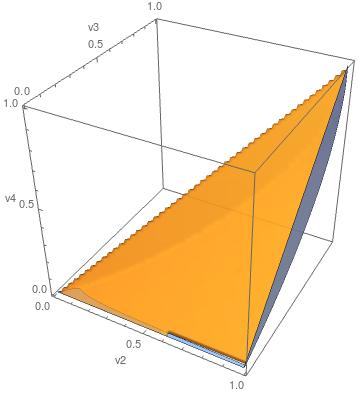}
\caption{the 3d region plot for the difference of two values in $[1,2,0,1]$ and $[0,1,1,2]$}\label{fig:f7}
\end{center}
\end{figure}

For $[1,2,1,0]$, we can find that the strategy for $a_1$ to decide whether to try to win or lose intentionally in the last match depends on the team weight vector $V=(1,v_2,v_3,v_4)$. Different team weight leads to different strategy. Next theorem will show the relationship between the strategy and vector $V$:

\begin{thm} \label{1210t}
Under a specific team weight vector $V=(1,v_2,v_3,v_4)$, $W=[1,2,1,0]$, then $a_1$ should try to win if and only if the following inequality holds. Otherwise, $a_1$ should lose intentionally.
\begin{align*} 
& 3v_2^2v_3^2v_4+v_2v_3^4+v_2v_3^3v_4+v_3^3v_4^2+4v_3^3v_4+4v_3^2v_4^2+5v_2v_3v_4 \\ & -4v_2^2v_3v_4^2-5v_2^2v_4^3-4v_3^4v_4-2v_2v_3v_4^2-v_3^2v_4-v_3v_4^2 \leq 0
\end{align*}
\end{thm}

The proof of theorem \ref{1210t} can be found in appendix $A.4$. The plot of this region is in figure \ref{fig:f8}, which is made by Mathematica.

\begin{figure}[h]
\begin{center}
\includegraphics[width=1.5in,height=1.5in]{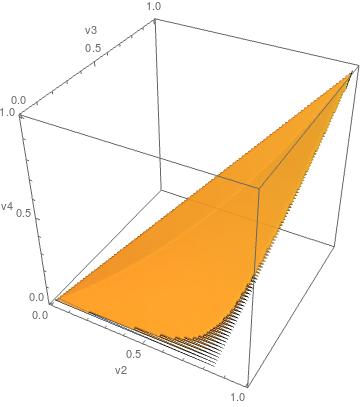}
\caption{the 3d region plot for the difference of two values in $[1,2,1,0]$}\label{fig:f8}
\end{center}
\end{figure}

\subsection{Analysis for the Second Week}
After completing the analysis of the last week, we want to analyze for the second week. Our idea is that given a winning vector $W^{*}$ at the beginning of the second week and the team weight vector $V$, if team $a_1$ tries to win, then four situations may happen on the second week: $a_1$ defeats $a_3$, $a_2$ defeats $a_4$ with probability $p_{13}p_{24}$; $a_1$ defeats $a_3$, $a_4$ defeats $a_2$ with probability $p_{13}p_{42}$; $a_3$ defeats $a_1$, $a_2$ defeats $a_4$ with probability $p_{31}p_{24}$; $a_3$ defeats $a_1$, $a_4$ defeats $a_2$ with probability $p_{31}p_{42}$. Assume these four situations bring $W^{*}$ to $W_1$, $W_2$, $W_3$, $W_4$ respectively, since we have completed the analysis of the last week, then we can get

\begin{align*} 
P_{\pi=(1,1,1,1)}^{(1)}(2,W^{*},V)&=p_{13}p_{24}\max\{P_{\pi=(1,1,1,1)}^{(1)}(3,W_1,V),P_{\pi=(0,1,1,1)}^{(1)}(3,W_1,V)\} \\ & +p_{13}p_{42}\max\{P_{\pi=(1,1,1,1)}^{(1)}(3,W_2,V),P_{\pi=(0,1,1,1)}^{(1)}(3,W_2,V)\} \\ & +p_{31}p_{24}\max\{P_{\pi=(1,1,1,1)}^{(1)}(3,W_3,V),P_{\pi=(0,1,1,1)}^{(1)}(3,W_3,V)\} \\ & +p_{31}p_{42}\max\{P_{\pi=(1,1,1,1)}^{(1)}(3,W_4,V),P_{\pi=(0,1,1,1)}^{(1)}(3,W_4,V)\}
\end{align*}

If $a_1$ loses intentionally in the second week, then two situations may happen on the second week: $a_3$ defeats $a_1$, $a_2$ defeats $a_4$ with probability $p_{24}$; $a_3$ defeats $a_1$, $a_4$ defeats $a_2$ with probability $p_{42}$. These two situations bring $W^{*}$ to $W_3$, $W_4$ respectively, then we can get

\begin{align*} 
P_{\pi=(0,1,1,1)}^{(1)}(2,W^{*},V)&=p_{24}\max\{P_{\pi=(1,1,1,1)}^{(1)}(3,W_3,V),P_{\pi=(0,1,1,1)}^{(1)}(3,W_3,V)\} \\ & +p_{42}\max\{P_{\pi=(1,1,1,1)}^{(1)}(3,W_4,V),P_{\pi=(0,1,1,1)}^{(1)}(3,W_4,V)\}
\end{align*}

\begin{thm} \label{secondwin}
Suppose the winning vector is $W^{*}$ at the beginning of the second week, then if
\begin{equation}
P_{\pi=(1,1,1,1)}^{(1)}(2,W^{*},V) \geq P_{\pi=(0,1,1,1)}^{(1)}(2,W^{*},V)
\end{equation}
$a_1$ should try to win in the second week. Otherwise, $a_1$ should lose intentionally. 
\end{thm}

Next, we take $W^{*}=[1,1,0,0]$ as an example.

\subsubsection{Example: $W^{*}=[1,1,0,0]$}
Suppose that after the first week, $W^{*}=[1,1,0,0]$, then if $a_1$ tries to win in the second week, after the second week, $W{*}$ may become the following four vectors:

\begin{equation}
\left\{ 
	\begin{array}{lr}
	$[2,2,0,0]$, & \text{with probability $p_{13}p_{24}$} \\
	$[2,1,0,1]$, & \text{with probability $p_{13}p_{42}$}	 \\
	$[1,2,1,0]$, & \text{with probability $p_{31}p_{24}$} \\
	$[1,1,1,1]$, & \text{with probability $p_{31}p_{42}$}	 \\
	\end{array}
\right.
\end{equation} 

Since we have already found the best strategy under $W_1=[2,2,0,0]$, $W_2=[2,1,0,1]$, $W_3=[1,2,1,0]$, $W_4=[1,1,1,1]$, recall that $P_{\pi=(1,1,1,1)}^{1}(2,W^{*},V)$ is the probability for $a_1$ to win the champion under $W^{*}=[1,1,0,0]$ at the beginning of the second week if $a_1$ tries to win in the second week, then by theorem \ref{secondwin} , we can get 

\begin{align*} 
P_{\pi=(1,1,1,1)}^{(1)}(2,[1,1,0,0],V)&=p_{13}p_{24}\max\{P_{\pi=(1,1,1,1)}^{(1)}(3,[2,2,0,0],V),P_{\pi=(0,1,1,1)}^{(1)}(3,[2,2,0,0],V)\} \\ & +p_{13}p_{42}\max\{P_{\pi=(1,1,1,1)}^{(1)}(3,[2,1,0,1],V),P_{\pi=(0,1,1,1)}^{(1)}(3,[2,1,0,1],V)\} \\ & +p_{31}p_{24}\max\{P_{\pi=(1,1,1,1)}^{(1)}(3,[1,2,1,0],V),P_{\pi=(0,1,1,1)}^{(1)}(3,[1,2,1,0],V)\} \\ & +p_{31}p_{42}\max\{P_{\pi=(1,1,1,1)}^{(1)}(3,[1,1,1,1],V),P_{\pi=(0,1,1,1)}^{(1)}(3,[1,1,1,1],V)\}
\end{align*}

Similarly, recall that $P_{\pi=(0,1,1,1)}^{1}(2,W^{*},V)$ is the probability for $a_1$ to win the champion under $W^{*}=[1,1,0,0]$ at the beginning of the second week if $a_1$ loses intentionally in the second week, note the $W^{*}$ may become the following two vectors:

\begin{equation}
\left\{ 
	\begin{array}{lr}
	$[1,2,1,0]$, & \text{with probability $p_{24}$} \\
	$[1,1,1,1]$, & \text{with probability $p_{42}$}	 \\
	\end{array}
\right.
\end{equation} 

Then, $P_{\pi=(0,1,1,1)}^{(1)}(2,W^{*},V)$ can be calculated by theorem \ref{secondwin}.

\begin{align*} 
P_{\pi=(0,1,1,1)}^{(1)}(2,[1,1,0,0],V)&=p_{24}\max\{P_{\pi=(1,1,1,1)}^{(1)}(3,[1,2,1,0],V),P_{\pi=(0,1,1,1)}^{(1)}(3,[1,2,1,0],V)\} \\ & +p_{42}\max\{P_{\pi=(1,1,1,1)}^{(1)}(3,[1,1,1,1],V),P_{\pi=(0,1,1,1)}^{(1)}(3,[1,1,1,1],V)\}
\end{align*}

Now, it is natural to compare the difference between $P_{\pi=(1,1,1,1)}^{(1)}(2,[1,1,0,0],V)$ and $P_{\pi=(0,1,1,1)}^{(1)}(2,[1,1,0,0],V)$ to get the best strategy in the second week. We will show the result in the following subsection.

\subsubsection{Analytical Results}
We assume that $v_1 = 1$, $v_2$, $v_3$, $v_4$ are in $(0,1)$, and $v_2 \geq v_3 \geq v_4$. We let $v_2$ be the $x$ axis, $v_3$ be the $y$ axis, $v_4$ be the $z$ axis. It is worth mentioning that in this section, we only provide the plots by Mathematica for the set $\{V=(v_2,v_3,v_4) | P_{\pi=(0,1,1,1)}(2,W,V)-P_{\pi=(1,1,1,1)}(2,W,V) \geq 0\}$ for given $W$ instead of giving an analytical formula for the region, because the analytical formula is consisted of several extremely complicated polynomial, which is hard to write it out explicitly.

For $W=[1,1,0,0]$, the result is in figure \ref{fig:f9}. The region in the plot is the set 

$\{V=(v_2,v_3,v_4) | P_{\pi=(0,1,1,1)}(2,W,V)-P_{\pi=(1,1,1,1)}(2,W,V) \geq 0\}$ where $W=[1,1,0,0]$. 

\begin{figure}[h]
\begin{center}
\includegraphics[width=1.5in,height=1.5in]{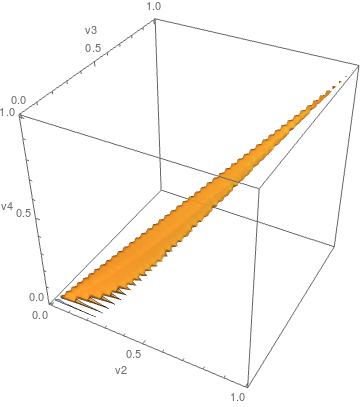}
\caption{values of $v_2$, $v_3$, $v_4$ such that $P_{\pi=(0,1,1,1)}(2,W,V)-P_{\pi=(1,1,1,1)}(2,W,V) \geq 0$ under $W=[1,1,0,0]$}\label{fig:f9}
\end{center}
\end{figure}

We can also get the region plot by Mathematica under $W=[1,0,1,0]$, $[0,1,0,1]$, $[0,0,1,1]$, the results are in figure \ref{fig:f99}.

\begin{figure}[h]
\includegraphics[width=1.5in,height=1.5in]{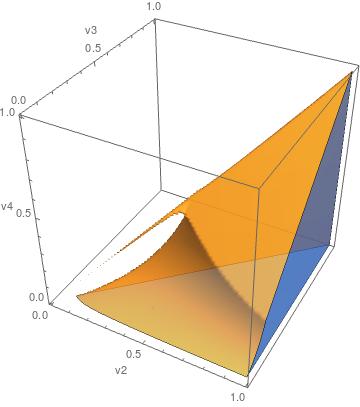}
\hspace{0.5in}
\includegraphics[width=1.5in,height=1.5in]{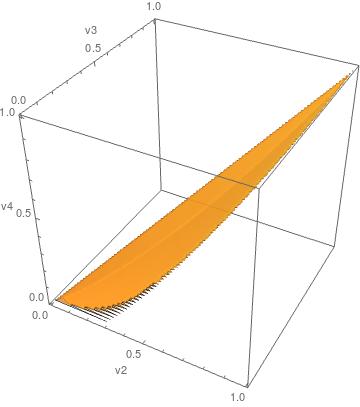}
\hspace{0.5in}
\includegraphics[width=1.5in,height=1.5in]{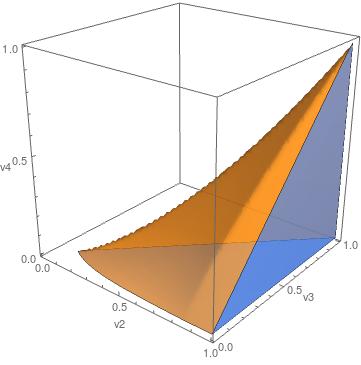}
\caption{Left:values of $v_2$, $v_3$, $v_4$ such that $P_{\pi=(0,1,1,1)}(2,W,V)-P_{\pi=(1,1,1,1)}(2,W,V) \geq 0$ under $W=[1,0,1,0]$. Middle:values of $v_2$, $v_3$, $v_4$ such that $P_{\pi=(0,1,1,1)}(2,W,V)-P_{\pi=(1,1,1,1)}(2,W,V) \geq 0$ under $W=[0,1,0,1]$. Right:values of $v_2$, $v_3$, $v_4$ such that $P_{\pi=(0,1,1,1)}(2,W,V)-P_{\pi=(1,1,1,1)}(2,W,V) \geq 0$ under $W=[0,0,1,1]$}\label{fig:f99}
\end{figure}

Thus, $a_1$ can decide to try to win or lose intentionally in the second week from the figure \ref{fig:f9}, \ref{fig:f99}. Now we move to analyze the strategy for $a_1$ in the first week!

\subsection{Analysis for the First Week}
After completing the analysis of the second week, we want to analyze for the first week. In the first week, if $a_1$ tries to win, winning vector $W$ may go to these four situations:

\begin{equation}
\left\{ 
	\begin{array}{lr}
	$[1,1,0,0]$, & \text{with probability $p_{14}p_{23}$} \\
	$[1,0,1,0]$, & \text{with probability $p_{14}p_{32}$}	 \\
	$[0,1,0,1]$, & \text{with probability $p_{41}p_{23}$} \\
	$[0,0,1,1]$, & \text{with probability $p_{41}p_{32}$}	 \\
	\end{array}
\right.
\end{equation} 

Since we have already found the best strategy under $[1,1,0,0]$, $[1,0,1,0]$, $[0,1,0,1]$, $[0,0,1,1]$, recall that $P_{\pi=(1,1,1,1)}^{1}(1,[0,0,0,0],V)$ is the probability for $a_1$ to win the champion at the beginning of the first week if $a_1$ tries to win, similar to the analysis for the second week, we can get: 

\begin{align*} 
P_{\pi=(1,1,1,1)}^{1}(1,[0,0,0,0],V)&=p_{14}p_{23}\max\{P_{\pi=(1,1,1,1)}^{(1)}(2,[1,1,0,0],V),P_{\pi=(0,1,1,1)}^{(1)}(2,[1,1,0,0],V)\} \\ & +p_{14}p_{32}\max\{P_{\pi=(1,1,1,1)}^{(1)}(2,[1,0,1,0],V),P_{\pi=(0,1,1,1)}^{(1)}(2,[1,0,1,1],V)\} \\ & +p_{41}p_{23}\max\{P_{\pi=(1,1,1,1)}^{(1)}(2,[0,1,0,1],V),P_{\pi=(0,1,1,1)}^{(1)}(2,[0,1,0,1],V)\} \\ & +p_{41}p_{32}\max\{P_{\pi=(1,1,1,1)}^{(1)}(2,[0,0,1,1],V),P_{\pi=(0,1,1,1)}^{(1)}(2,[0,0,1,1],V)\}
\end{align*}

If $a_1$ loses intentionally, winning vector $W$ may go to these two situations:

\begin{equation}
\left\{ 
	\begin{array}{lr}
	$[0,1,0,1]$, & \text{with probability $p_{23}$} \\
	$[0,0,1,1]$, & \text{with probability $p_{32}$}	 \\
	\end{array}
\right.
\end{equation} 

Then, $P_{\pi=(0,1,1,1)}^{(1)}(1,[0,0,0,0],V)$ can be written as

\begin{align*} 
P_{\pi=(0,1,1,1)}^{(1)}(1,[0,0,0,0],V)&=p_{23}\max\{P_{\pi=(1,1,1,1)}^{(1)}(2,[0,1,0,1],V),P_{\pi=(0,1,1,1)}^{(1)}(2,[0,1,0,1],V)\} \\ & +p_{32}\max\{P_{\pi=(1,1,1,1)}^{(1)}(2,[0,0,1,1],V),P_{\pi=(0,1,1,1)}^{(1)}(2,[0,0,1,1],V)\}
\end{align*}

Next we can compare the difference and due to the huge amount of calculation, the Mathematica cannot provide a nice 3d region plot, instead, we use Python to make the 3d scatter plot, which can also provide some information of the rough shape of the region. We let $v_2$ be the $x$ axis, $v_3$ be the $y$ axis, $v_4$ be the $z$ axis. The points in the plots is the set of point $\{V=(v_2,v_3,v_4) | P_{\pi=(0,1,1,1)}^{1}(1,[0,0,0,0],V)-P_{\pi=(1,1,1,1)}^{1}(1,[0,0,0,0],V) \geq 0\}$ for given $W$. Here is the result.

\begin{figure}[h]
\begin{center}
\includegraphics[width=3in,height=2.5in]{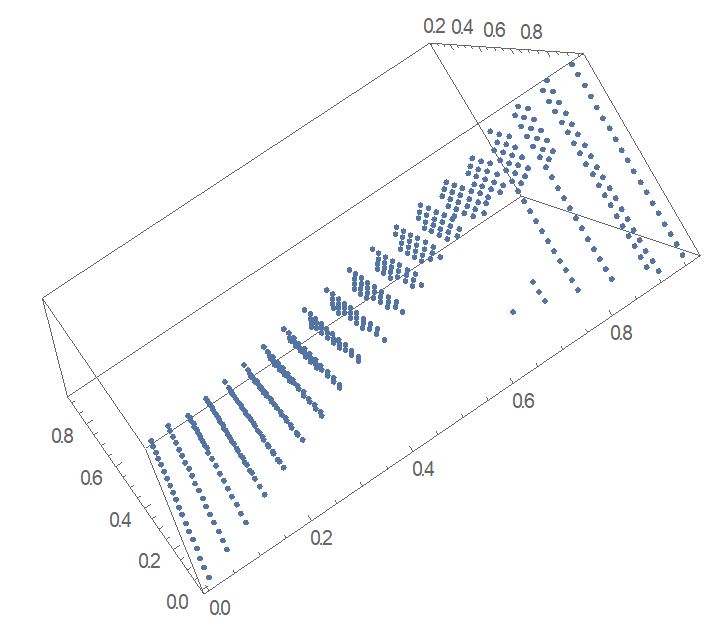}
\caption{the list points of $[v_2, v_3, v_4]$ when it is better for $a_1$ to lose intentionally in week $1$}\label{fig:10}
\end{center}
\end{figure}

Thus, $a_1$ can investigate the team weight of $a_2$, $a_3$, $a_4$ at the beginning of the first week. If their team weight is in figure \ref{fig:10}, $a_1$ should lose intentionally in the first week. Otherwise, $a_1$ should try to win. Hence, we have fully analyzed this FRNS model and in the next part we will analyze the FRS model.

\section{Analysis for the FRS Model}
\subsection{Description and Assumption of FRS Model}
In this part, all assumptions are as same as the model in previous section except for the strategies of the other three teams. Recall that in the FRNS model, we assume that other teams try to win every game, i.e. $\pi_2=\pi_3=\pi_4=1$. However, in this part, we assume that other three teams are smart enough to lose some games intentionally to maximize their winning probability of the tournament. Hence, the strategy for one particular game may not be pure strategy. Instead, every team may have a mixed strategy for each game. In this model, $\pi_i \in [0,1]$ for $i \in \{1,2,3,4\}$. Note that the boundary point $0$ represents the strategy to lose intentionally and $1$ represents the strategy to try to win. If $\pi_i \notin \{0,1\}$, it means team $a_i$ tries to win with probability $\pi_i$ and loses intentionally with probability $1-\pi_i$. We assume that while two teams $a_i$ and $a_j$ are playing a game, if both of them try to win, then $a_i$ has probability $p_{ij}=\frac{v_i}{v_i+v_j}$ to win. If one of them tries to win and the other loses intentionally, then we assume that the one who tries to win will have 100\% probability to win this game. If both of them loses intentionally, then the game is decided by flipping a fair coin. In the game theory problem with mixed strategy, we usually have to find the Nash equilibrium. The Nash equilibrium is a concept of game theory where the optimal outcome of a game is one where no player has an incentive to deviate from his chosen strategy after considering an opponent's choice. Hence, in this model, in week $m$, given winning vector $W$ and weight vector $V$, $\pi=(\pi_1,\pi_2,\pi_3,\pi_4)$ is a Nash equilibrium if for each $i \in \{1,2,3,4\}$, $P_\pi^{(i)}(m,W,V) \geq P_{\pi^{'}}^{(i)}(m,W,V)$ for $\forall \pi^{'}$ s.t. $\pi_j^{'}=\pi_j$,$j \neq i$.

\subsection{Example $W=[1,2,0,1]$}
\subsubsection{Analysis for $(\pi_1, \pi_2, \pi_3, \pi_4)$}
Similar with the analysis of FRNS model, we still first analyze the last game. Take $W=[1,2,0,1]$ as an example, recall that in our FRNS model, we denote $T_A$, $T_B$, $T_C$ as the probability for $a_1$ to win the champion in tournament $A$, $B$, $C$ respectively. In this section, we define $T_{ij}$ as the probability team $a_i$ to win the champion in the tournament $j$, where $i \in \{1,2,3,4\}$, $j \in \{A,B,C\}$. By Theorem \ref{thm:ordertourn}, we know that $T_{1A} \geq T_{1B} \geq T_{1C}$ when $v_1 \geq v_2 \geq v_3 \geq v_4$. Since in our FRS model, we not only analyze the strategy of $a_1$, but also analyze the strategies of $a_2$, $a_3$, $a_4$. Next, we will introduce a lemma to show the relationship between $T_{ij}$.

\begin{lemA} \label{T1}
If $v_1 \geq v_2 \geq v_3 \geq v_4$, then
\begin{equation}
T_{1A} \geq T_{1B} \geq T_{1C}
\end{equation}
\begin{equation}
T_{2B} \geq T_{2A} \geq T_{2C}
\end{equation}
\begin{equation}
T_{3C} \geq T_{3A} \geq T_{3B}
\end{equation}
\begin{equation}
T_{4C} \geq T_{4B} \geq T_{4A}
\end{equation}
\end{lemA}

The proof of lemma \ref{T1} can be found in the appendix $A.5$. We just simply calculate all these variables and find the difference between each of them. 

In addition, we assume that in week $3$, $a_1$ vs $a_2$, $a_3$ vs $a_4$ happen simultaneously. Thus, before playing the game, any team cannot know the result of the other game. We introduce two variables: $A_{12}$, $A_{34}$, which represent the real probability for $a_1$ defeats $a_2$, and $a_3$ defeats $a_4$. The real probability of $A_{12}$ is the sum of winning probability for $a_1$ under four situations: $a_1$ tries to win 
and $a_2$ tries to win, $a_1$ tries to win and $a_2$ loses intentionally, $a_1$ loses intentionally and $a_2$ tries to win, $a_1$ loses intentionally and $a_2$ loses intentionally. Recall that in this section, we still assume that if both team lose intentionally, we can flip a fair coin to decide who wins the game. Thus, the formula of $A_{12}$ is: 
\begin{equation}
A_{12}=v_{12}\pi_1\pi_2+\pi_1(1-\pi_2)+\frac{1}{2}(1-\pi_1)(1-\pi_2)
\end{equation}
Similarly, 
\begin{equation}
A_{34}=v_{34}\pi_3\pi_4+\pi_3(1-\pi_4)+\frac{1}{2}(1-\pi_3)(1-\pi_4)
\end{equation}

Next, we define the payoff function as the probability to win the champion under given strategy. Define $Q_{ij}$, $i \in \{1,2,3,4\}$, $j \in \{a,b,c,d\}$ as the probability for $a_i$ to win the tournament under condition $j$, such that $Q_{1a}$, $Q_{2a}$, represents the probability for $a_1$, $a_2$ to win the champion if both $a_1$ and $a_2$ tries to win, respectively. $Q_{1b}$, $Q_{2b}$, represents the probability for $a_1$, $a_2$ to win the champion if $a_1$ loses intentionally and $a_2$ tries to win, respectively. $Q_{1c}$, $Q_{2c}$, represents the probability for $a_1$, $a_2$ to win the champion if $a_1$ tries to win and $a_2$ loses intentionally, respectively. $Q_{1d}$, $Q_{2d}$, represents the probability for $a_1$, $a_2$ to win the champion if both $a_1$ and $a_2$ loses intentionally, respectively. Similarly, we define $Q_{3a}$, $Q_{4a}$ as the probability for $a_3$, $a_4$ to win the champion if both $a_3$ and $a_4$ tries to win, respectively. $Q_{3b}$, $Q_{4b}$, represents the probability for $a_3$, $a_4$ to win the champion if $a_3$ loses intentionally and $a_4$ tries to win, respectively. $Q_{3c}$, $Q_{4c}$, represents the probability for $a_3$, $a_4$ to win the champion if $a_3$ tries to win and $a_4$ loses intentionally, respectively. $Q_{3d}$, $Q_{4d}$, represents the probability for $a_3$, $a_4$ to win the champion if both $a_3$ and $a_4$ loses intentionally, respectively.

To calculate the payoff functions, we show some examples. For example, to calculate $Q_{1a}$, we notice that if $a_3$ defeats $a_4$, then $W=[1,2,1,1]$. If both $a_1$ and $a_2$ tries to win and if $a_1$ wins, then $W=[2,2,1,1]$, recall our flipping coin assumption, $a_1$ has $\frac{1}{2} (T_{1A}+T_{1B} )$ probability to win the champion. If $a_2$ wins, then $W=[1,3,1,1]$, $a_1$ has $\frac{1}{3} (T_{1A}+T_{1B}+T_{1C} )$ probability to win the champion. Otherwise, if $a_4$ defeats $a_3$, then $W=[1,2,0,2]$, then if $a_1$ wins, $W=[2,2,0,2]$, $a_1$ has $\frac{1}{3} (T_{1A}+T_{1B}+T_{1C} )$ probability to win the champion. If $a_2$ wins, $W=[1,3,0,2]$, $a_1$ will enter tournament $A$, and has $T_{1A}$ probability to win the champion. Thus, 
\begin{equation}
Q_{1a}=A_{34}(v_{12}\frac{1}{2} (T_{1A}+T_{1B} )+(1-v_{12} )\frac{1}{3} (T_{1A}+T_{1B}+T_{1C} ))+(1-A_{34} )(v_{12}\frac{1}{3}((T_{1A}+T_{1B}+T_{1C} )+(1-v_{12})T_{1A})
\end{equation}

Notice that if $i=1,2$, $Q_{ij}$ is a function with parameter $v_1$, $v_2$, $v_3$, $v_4$, $A_{34}$. Since $A_{34}$ is a function with parameter $\pi_3$, $\pi_4$, 
\begin{equation}
Q_{ij} = F_{ij}(v_1, v_2, v_3, v_4, \pi_3, \pi_4)
\end{equation}
Similarly, if $i = 3,4$, 
\begin{equation}
Q_{ij} = F_{ij}(v_1, v_2, v_3, v_4, \pi_1, \pi_2)
\end{equation}
\begin{figure}
\begin{center}
\includegraphics[width=2in,height=2in]{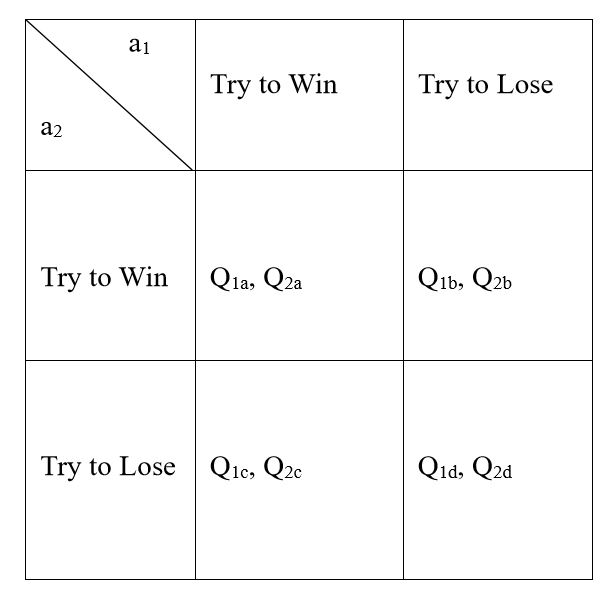}
\hspace{0.5in}
\includegraphics[width=2in,height=2in]{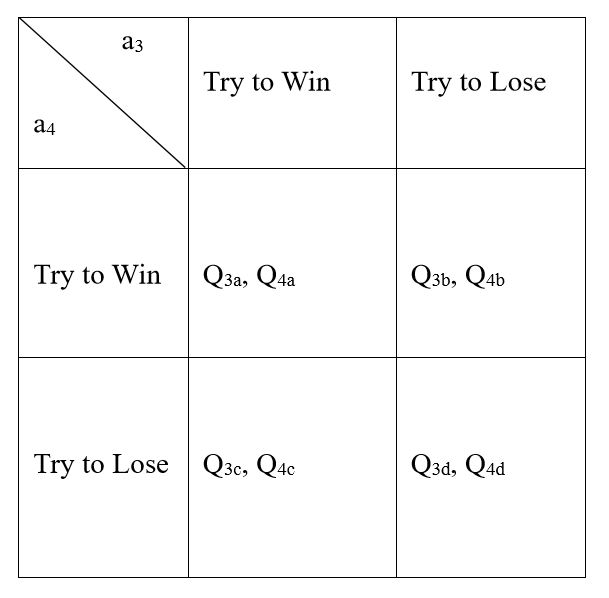}
\caption{The game theory table for four teams}\label{fig:game1}
\end{center}
\end{figure}

Now, we want to introduce our method to calculate $\pi=(\pi_1,\pi_2,\pi_3,\pi_4)$. In our algorithm, the main logic is to write the probability for each team to win the champion as a function with variable $\pi$ and $Q_{ij}$. By the definition of the Nash equilibrium, if all teams own mixed strategy, then the partial derivatives of probability for team $a_i$ to win the champion with respect to the variable $\pi_i$ all equal to $0$. If not, the teams will have pure strategy. Thus, we calculate the partial derivative of the probability for each team to win the champion respect to the team $i$'s strategy $\pi_i$ and get the solutions of $\pi$. We will expand the analysis of how to solve $\pi$ now. Recall that in our main logic, firstly, we have to find the probability of each team to win the champion. Let $E_1$, $E_2$, $E_3$, $E_4$ be probability of team $a_1$, $a_2$, $a_3$, $a_4$ to win the tournament, respectively. Recall that $Q_{ij}$, $i \in \{1,2,3,4\}$, $j \in \{a,b,c,d\}$ is the probability for $a_i$ to win the champion under condition $j$. Define $P(j)$ as the probability of event $j$. Here for team $a_i$, we have 
\begin{equation} \label{esum}
E_i=\sum_{j \in \{a,b,c,d\}}Q_{ij}P(j)
\end{equation}

For example, recall that $Q_{1a}$, $Q_{1b}$, $Q_{1c}$, $Q_{1d}$ represents the probability for $a_1$ to win the tournament if both $a_1$ and $a_2$ tries to win, if $a_1$ loses intentionally and $a_2$ tries to win, if $a_1$ tries to win and $a_2$ loses intentionally, and if both $a_1$ and $a_2$ loses intentionally, respectively. Then by equation \ref{esum}, we have $E_1=\pi_1\pi_2Q_{1a}+(1-\pi_1)\pi_2Q_{1b}+\pi_1(1-\pi_2)Q_{1c}+(1-\pi_1)(1-\pi_2)Q_{1d}$. Similarly, we can write the following equation system.

\begin{equation} \label{expectation}
\left\{ \begin{array}{lr}  E_1=\pi_1\pi_2Q_{1a}+(1-\pi_1)\pi_2Q_{1b}+\pi_1(1-\pi_2)Q_{1c}+(1-\pi_1)(1-\pi_2)Q_{1d}, & \\ E_2=\pi_1\pi_2Q_{2a}+(1-\pi_1)\pi_2Q_{2b}+\pi_1(1-\pi_2)Q_{2c}+(1-\pi_1)(1-\pi_2)Q_{2d}, \\ E_3=\pi_3\pi_4Q_{3a}+(1-\pi_3)\pi_4Q_{3b}+\pi_3(1-\pi_4)Q_{3c}+(1-\pi_3)(1-\pi_4)Q_{3d}, \\ E_4=\pi_3\pi_4Q_{4a}+(1-\pi_3)\pi_4Q_{4b}+\pi_3(1-\pi_4)Q_{4c}+(1-\pi_3)(1-\pi_4)Q_{4d}, &    \end{array}    \right.  \end{equation}

Secondly, we want to see how each team's strategy effects their winning probability. We calculate the partial derivative for each probability in equation system \ref{expectation} with respect to the variable $\pi_i$, recall that $Q_{1j}$, $Q_{2j}$ only depend on $\pi_3$ and $\pi_4$, $Q_{3j}$, $Q_{4j}$ only depend on $\pi_1$ and $\pi_2$, then, we can get 
\begin{equation} \label{pderivE}
\left\{ \begin{array}{lr}
\frac{\partial E_1}{\partial \pi_1}=\pi_2(Q_{1a}-Q_{1b}-Q_{1c}+Q_{1d})+Q_{1c}-Q_{1d}, & \\  \frac{\partial E_2}{\partial \pi_2}=\pi_1(Q_{2a}-Q_{2b}-Q_{2c}+Q_{2d})+Q_{2b}-Q_{2d}, \\ \frac{\partial E_3}{\partial \pi_3}=\pi_4(Q_{3a}-Q_{3b}-Q_{3c}+Q_{3d})+Q_{3c}-Q_{3d}, \\ \frac{\partial E_4}{\partial \pi_4}=\pi_3(Q_{4a}-Q_{4b}-Q_{4c}+Q_{4d} )+Q_{4b}-Q_{4d},  &   \end{array}    \right.  
\end{equation}

From equation system \ref{pderivE}, we notice that $\frac{\partial E_i}{\partial \pi_i}$ is linearly related to the strategy variable $\pi_j$, where $a_j$ is the rival of $a_i$ in the last week. Hence, we can draw the following conclusion:

\begin{thm}
If $\frac{\partial E_i}{\partial \pi_i} \neq 0$, then team $a_i$ should make a pure strategy. Moreover, if $\frac{\partial E_i}{\partial \pi_i} > 0$, team $a_i$ should try to win. If $\frac{\partial E_i}{\partial \pi_i} < 0$, team $a_i$ should lose intentionally.
\end{thm}

\begin{proof}
Notice that if $\frac{\partial E_i}{\partial \pi_i} > 0$, then $\frac{\partial E_i}{\partial \pi_i}$ is a linear function with positive slope. We get the maximum of $E_i$ if $\pi_i$ achieves its maximum, which is $1$. Similarly, if $\frac{\partial E_i}{\partial \pi_i} < 0$, then we get the maximum of $E_i$ if $\pi_i$ achieves its minimum, which is $0$. 
\end{proof}

Now, we focus on finding the Nash equilibrium. We first apply the Nash's Existence Theorem to show the existence of the Nash equilibrium $\pi$.

\begin{thm} [Nash's Existence Theorem] \label{exis}
\citep{MR46638} If we allow mixed strategies (where a pure strategy is chosen at random, subject to some fixed probability), then every game with a finite number of players in which each player can choose from finitely many pure strategies has at least one Nash equilibrium.
\end{thm}

In our problem setting, the number of players is finite, each player has two pure strategies: win or lose. Hence, we know that the Nash equilibrium exists in our problem setting. The next proposition shows that the teams can have mixed strategies if and only if all partial derivatives of winning probabilities equals to zero. 

\begin{prop} \label{propp}
For all $m$, $W$, and possible $V$, $\pi=(\pi_1,\pi_2,\pi_3,\pi_4)$ is the Nash equilibrium if and only if
\begin{equation} \left\{ \begin{array}{lr}  \frac{\partial E_1}{\partial \pi_1}=0, \\ \frac{\partial E_2}{\partial \pi_2}=0, \\\frac{\partial E_3}{\partial \pi_3}=0, \\ \frac{\partial E_4}{\partial \pi_4}=0,   &   \end{array}    \right.  \end{equation} \label{par}
\end{prop}

Recall that we have mentioned that at this moment, we do not know the value of $\pi=(\pi_1,\pi_2,\pi_3,\pi_4)$, hence we solve the equation system in proposition \ref{par} without expanding $A_{12}$ and $A_{34}$, we can get 

\begin{equation} \label{prob}
\left\{ \begin{array}{lr}  A_{34}=\frac{4T_{1A}-2T_{1B}-2T_{1C}}{5T_{1A}-T_{1B}-4T_{1C}}, \\ A_{34}=\frac{4T_{2A}-2T_{2B}-2T_{2C}}{5T_{2A}-T_{2B}-4T_{2C}}, \\ A_{12}=\frac{4T_{3A}-2T_{3C}-2T_{3B}}{5T_{3A}-T_{3B}-4T_{3C}}, \\ A_{12}=\frac{4T_{4A}-2T_{4C}-2T_{4B}}{5T_{4A}-T_{4B}-4T_{4C}},  &   \end{array}    \right.  \end{equation}  

From equation system \ref{prob}, we know that the Nash equilibrium follows 
\begin{equation} \label{e1}
A_{34}=\frac{4T_{1A}-2T_{1B}-2T_{1C}}{5T_{1A}-T_{1B}-4T_{1C}}=\frac{4T_{2A}-2T_{2B}-2T_{2C}}{5T_{2A}-T_{2B}-4T_{2C}}
\end{equation}
\begin{equation} \label{e2}
A_{12}=\frac{4T_{3A}-2T_{3C}-2T_{3B}}{5T_{3A}-T_{3B}-4T_{3C}}=\frac{4T_{4A}-2T_{4C}-2T_{4B}}{5T_{4A}-T_{4B}-4T_{4C}}
\end{equation}

We have to solve these two equations. We know that the strategy is related to all team weights. Next, we claim the following theorem:

\begin{thm}\label{solution}
The only solution of \ref{e1}, \ref{e2} is that $v_1=v_2=1$, $v_3=v_4\leq 1$.
\end{thm}

By theorem \ref{solution}, we know that the only situation where all teams have mixed strategies is when $a_2$ is as strong as $a_1$, $a_3$ is as strong as $a_4$. Next, we will give a proof of theorem \ref{solution}.

\begin{proof}
Equation \ref{e1} implies that 
\begin{equation} \label{e3}
T_{1B}T_{2C}-T_{1B}T_{2A}+T_{1C}T_{2A}+T_{1A}T_{2B}-T_{1A}T_{2C}-T_{1C}T_{2B}=0
\end{equation}

We can write \ref{e3} as 
\begin{equation} \label{e4}
T_{1A}(T_{2B}-T_{2C})+T_{1B}(T_{2C}-T_{2A} )+T_{1C}(T_{2A}-T_{2B})=0
\end{equation}

To solve \ref{e4}, we have to introduce the Chebyshev's sum inequality \citep{MR0046395}

\begin{lemA}[Chebyshev's sum inequality]
If $a_1 \geq a_2 \geq ... \geq a_n$, $b_1 \geq b_2 \geq ... \geq b_n$, then 
\begin{equation}
\frac{1}{n}\sum_{k=1}^n a_kb_k \geq (\frac{1}{n}\sum_{k=1}^n a_k)(\frac{1}{n}\sum_{k=1}^n b_k)
\end{equation}
\end{lemA}

Recall lemma \ref{T1}, $T_{1A}\geq T_{1B}\geq T_{1C}$,$T_{2B}\geq T_{2A}\geq T_{2C}$, if $(T_{2C}-T_{2A})\geq (T_{2A}-T_{2B})$, then by Chebyshev's sum inequality, 
\begin{equation} \label{e5}
T_{1B}(T_{2C}-T_{2A} )+T_{1C}(T_{2A}-T_{2B}) \geq \frac{1}{2}(T_{1B}+T_{1C})(T_{2C}-T_{2A}+T_{2A}-T_{2B})=\frac{1}{2}(T_{1B}+T_{1C} )(T_{2C}-T_{2B})
\end{equation}

Applying \ref{e5} to \ref{e4}, we get 
\begin{equation}
0\geq T_{1A}(T_{2B}-T_{2C})+\frac{1}{2}(T_{1B}+T_{1C} )(T_{2C}-T_{2B})=(T_{1A}-\frac{T_{1B}+T_{1C}}{2})(T_{2B}-T_{2C})
\end{equation}

This only happens when $T_{1A}=T_{1B}=T_{1C}$. Otherwise, 
\begin{equation} \label{e6}
T_{2C}-T_{2A} \leq T_{2A}-T_{2B}
\end{equation}

We can also write equation \ref{e3} as $T_{2B}(T_{1A}-T_{1C})+T_{2A}(T_{1C}-T_{1B})+T_{2C}(T_{1B}-T_{1A})=0$. By applying Chebyshev's sum inequality, we know that if $T_{1C}-T_{1B} \geq  T_{1B}-T_{1A}$, the only possible solution is that $T_{2A}=T_{2B}=T_{2C}$. Otherwise,

\begin{equation} \label{e7}
T_{1C}-T_{1B} \leq T_{1B}-T_{1A}
\end{equation}

Similarly, we apply the same method to equation \ref{e2}. We get 
\begin{equation} \label{e8}
T_{4A}-T_{4B} \leq T_{4B}-T_{4C}
\end{equation}
\begin{equation} \label{e9}
T_{3B}-T_{3A} \leq T_{3A}-T_{3C}
\end{equation}

By solving \ref{e6}, \ref{e7}, \ref{e8}, \ref{e9}, we get that the only solution is that $v_1=v_2=1$, $v_3=v_4\leq 1$.
\end{proof}

Now, we know that if $v_1=v_2=1$, $v_3=v_4\leq 1$, every team has mixed strategy. Next, we want to find what the mixed strategy is. In this situation,
 
\begin{equation} \label{e10}
T_{1A}= T_{1B}= T_{2A}= T_{2B} \geq T_{1C}= T_{2C}
\end{equation}
\begin{equation} \label{e11}
T_{3C}= T_{4C}=1-T_{1C}
\end{equation}
\begin{equation} \label{e12}
T_{3A}= T_{3B}= T_{4A}= T_{4B}=1-T_{1A}
\end{equation}
We re-calculate $A_{12}$ and $A_{34}$ by plugging \ref{e10}, \ref{e11}, \ref{e12} to equation \ref{e1}, \ref{e2}, we can get 
\begin{equation} \label{e13}
A_{34}=\frac{4T_{1A}-2T_{1B}-2T_{1C}}{5T_{1A}-T_{1B}-4T_{1C}}=\frac{2T_{1A}-2T_{1C}}{4T_{1A}-4T_{1C}}=\frac{1}{2}
\end{equation}

Similarly, $A_{12} = \frac{1}{2}$. 

\begin{cor} \label{xyzw}
Under the situation where $v_1=v_2=1$, $v_3=v_4\leq 1$, the mixed strategy $(\pi_1,\pi_2,\pi_3,\pi_4)$ satisfies $\pi_1=\pi_2$, $\pi_3=\pi_4$
\end{cor}

\begin{proof}
By expanding $A_{34}$, 
\begin{equation}
A_{34}=v_{34}\pi_3\pi_4+\pi_3(1-\pi_4)+\frac{1}{2}(1-\pi_3)(1-\pi_4)=\pi_3\pi_4+2\pi_3-2\pi_3\pi_4+\frac{1}{2}-\pi_3-\pi_4+\pi_3\pi_4=\frac{1}{2}
\end{equation}
We can get $\pi_3 = \pi_4$, similarly, by expanding $A_{12}$, we can also get $\pi_1 = \pi_2$.
\end{proof}

Next, we re-calculate the probability for each team to win the champion, $E_i$, $i=\{1,2,3,4\}$, under $\pi_1=\pi_2$, $\pi_3=\pi_4$. 

\begin{cor}
Under $\pi_1=\pi_2$, $\pi_3=\pi_4$, the probability for each team to win the champion, $E_i$, does not depend on $\pi=(\pi_1,\pi_2,\pi_3,\pi_4)$.
\end{cor}

\begin{proof}
\begin{align*} \label{e14}
E_1&=\pi_1\pi_2Q_{1a}+(1-\pi_1)\pi_2Q_{1b}+\pi_1(1-\pi_2)Q_{1c}+(1-\pi_1)(1-\pi_2)Q_{1d}\\ &=\pi_1^2(Q_{1a}-Q_{1b}-Q_{1c}+Q_{1d})+\pi_1(Q_{1b}+Q_{1c}-2Q_{1d})+Q_{1d}
\end{align*}
Due to $T_{1A}= T_{1B}$, also by equation \ref{e13} we know $A_{12}=A_{34}=\frac{1}{2}$, we can get 
\begin{equation} \label{e15}
Q_{1a}-Q_{1b}-Q_{1c}+Q_{1d}=0
\end{equation}
\begin{equation} \label{e16}
Q_{1b}+Q_{1c}-2Q_{1d}=0
\end{equation}
Thus, $E_1=Q_{1d}$. Similarly, we can get $E_2=Q_{2d}$,$E_3=Q_{3d}$,$E_4=Q_{4d}$. This proves that $E_i$ is independent with the strategy variable $\pi$. 
\end{proof}

Hence, we can draw the following conclusion:

\begin{thm}{Conclusion} \label{concl}
\\
If $v_1=v_2=1$, $v_3=v_4\leq 1$, then all $\pi_1=\pi_2$, $\pi_3=\pi_4$ are Nash equilibriums. Otherwise, all teams should use pure strategy.
\end{thm}

Now, we want to raise a question: How to compute the Nash equilibrium if $\pi_1 \neq \pi_2$ or $\pi_3 \neq \pi_4$? From theorem \ref{concl}, we know that all teams have pure strategy if $\pi_1 \neq \pi_2$ or $\pi_3 \neq \pi_4$, i.e. $\pi \in \{0,1\}^4$. Since $card(\{0,1\}^4)=2^4=16$, we can plug all these 16 possible $\pi$ to the verification process: given a vector $\pi$, we can calculate all $Q_{ij}$ and draw the game theory table. Then, we can check whether the game theory problem has dominant strategy or not. By theorem \ref{exis}, we know that there exists some solutions among the 16 possible $\pi$.

If we are the coach of a sport team, to determine whether to try to win the next game or to lose intentionally, we can use our algorithm to get the answer. We first approximate the team weights $v_1$, $v_2$, $v_3$, $v_4$ by previous data. If the win vector before the next game is $W$, by applying our algorithm, we can get the strategy $\pi=(\pi_1,\pi_2,\pi_3,\pi_4)$ and the probability of winning the champion of each team $E_1$, $E_2$, $E_3$, $E_4$. 

\section{Future Works}
One important question is that whether theorem \ref{concl} holds when there are eight teams. A big difference between the four-team model and eight-team model is that under the four-team model, lemma \ref{T1} holds for any $v_1 \geq v_2 \geq v_3 \geq v_4$. However, under the eight-team model, according to the model created by J. Schwenk \citep{schwenk2000correct}, he showed that under some specific situations, the second seed has larger probability than the first seed to win the tournament, which implies that we cannot generalize lemma \ref{T1} for the eight-team model. Hence, one of the future work is to find the Nash equilibrium under the eight-team model. One possible method is following the definition of Nash equilibrium, proposition \ref{propp}, to solve the equation system. However, without lemma \ref{T1}, we may not draw the conclusion that all teams should use pure strategy for most cases.

\bibliographystyle{plain}
\bibliography{research_essay_0521}

\appendix
\appendixpage
\addappheadtotoc

\section{Proofs for FRNS Model}
\subsection{Proof of Theorem \ref{thm:ordertourn}}
\begin{proof}
\begin{align*}
 T_A &= \frac{v_1}{v_1+v_4} \left(\frac{v_1}{v_1+v_2}\frac{v_2}{v_2+v_3}+\frac{v_1}{v_1+v_3}\frac{v_3}{v_2+v_3} \right)=\frac{v_1^2 [v_2 (v_1+v_3 )+v_3 (v_1+v_2 )]}{(v_1+v_4 )(v_1+v_2 )(v_1+v_3 )(v_2+v_3 )} \\
 T_B &= \frac{v_1}{v_1+v_3} \left(\frac{v_1}{v_1+v_2} \frac{v_2}{v_2+v_4}+\frac{v_1}{v_1+v_4} \frac{v_4}{v_2+v_4} \right)=\frac{v_1^2 [v_2 (v_1+v_4 )+v_3 (v_2+v_4 )]}{(v_1+v_4 )(v_1+v_2 )(v_1+v_3 )(v_2+v_4 )} \\
 T_C &= \frac{v_1}{v_1+v_2} \left(\frac{v_1}{v_1+v_3} \frac{v_3}{v_3+v_4}+\frac{v_1}{v_1+v_4} \frac{v_4}{v_3+v_4}\right)=\frac{v_1^2 [v_3 (v_1+v_4 )+v_4 (v_1+v_3 )]}{(v_1+v_4 )(v_1+v_2 )(v_1+v_3 )(v_3+v_4 )} 
\end{align*}
		
Compare $P_A$ and $P_B$, we compute 
\begin{align*}  
T_A-T_B &= \left(\frac{v_1^2}{(v_1+v_4 )(v_1+v_2 )(v_1+v_3)} \right)\left(\frac{v_2 (v_1+v_3 )+v_3 (v_1+v_2 )}{v_2+v_3}-\frac{v_2 (v_1+v_4 )+v_3 (v_2+v_4 )}{v_2+v_4} \right)\\
&=\left(\frac{v_1^2}{(v_1+v_4 )(v_1+v_2 )(v_1+v_3)} \right)(v_1 v_2+2v_2 v_3+v_1 v_3 )(v_2+v_4 )-(v_1 v_2+v_2 v_3+v_2 v_4+v_3 v_4 )(v_2+v_3 )\\&=\left(\frac{v_1^2}{(v_1+v_4 )(v_1+v_2 )(v_1+v_3)} \right)(2v_2 v_3^2-2v_3^2 v_4)=\left(\frac{v_1^2}{(v_1+v_4 )(v_1+v_2 )(v_1+v_3)} \right)(2v_3^2 (v_2-v_4))
\end{align*}
$\because$ $v_1 \geq v_2 \geq v_3 \geq v_4$, $\therefore$ $T_A-T_B\geq 0$ 

Then, compare T$_B$ and T$_C$, we compute
\begin{align*}
T_B-T_C&= \left(\frac{v_1^2}{(v_1+v_4 )(v_1+v_2 )(v_1+v_3 )} \right)\left(\frac{v_2 (v_1+v_4 )+v_3 (v_2+v_4 )}{v_2+v_4}-\frac{v_3 (v_1+v_4 )+v_4 (v_1+v_3 )}{v_3+v_4} \right)\\&= \left(\frac{v_1^2}{(v_1+v_4 )(v_1+v_2 )(v_1+v_3 )} \right)\left((v_1 v_2+v_2 v_4+v_2 v_3+v_3 v_4 )(v_3+v_4 )-(v_1 v_3+v_1 v_4+2v_3 v_4 )(v_2+v_4 ) \right)\\&=\left(\frac{v_1^2}{(v_1+v_4 )(v_1+v_2 )(v_1+v_3 )} \right)\left((v_3^2 v_4-v_3 v_4^2 )+(v_2 v_3^2-v_2 v_3 v_4 )+(v_1 v_2 v_4-v_1 v_3 v_4) \right) \\&= \left(\frac{v_1^2}{(v_1+v_4 )(v_1+v_2 )(v_1+v_3 )} \right)\left(v_3v_4(v_3-v_4)+v_2v_3(v_2-v_4)+v_1v_4(v_2-v_3) \right)
\end{align*}
$\because$ $v_1 \geq v_2 \geq v_3 \geq v_4$, $\therefore$ $T_B-T_C\geq 0$ 

Thus, $T_A\geq T_B\geq T_C$ for any $v_1$, $v_2$, $v_3$, $v_4$, subject to $v_1 \geq v_2 \geq v_3 \geq v_4$.
\end{proof}

\subsection{Proof of Theorem \ref{alwayswin}}
\begin{proof}
\begin{align*}
P_{\pi=(1,1,1,1)}^{(1)}(3,W,V)-P_{\pi=(0,1,1,1)}^{(1)}(3,W,V)&=\left(p_{12}p_{34}T_A+p_{12}p_{43}T_B+p_{21}p_{34}T_B+p_{21}p_{43}T_A \right)\\&-\left(p_{34}T_A+p_{43}T_B \right)\\&= T_Ap_{12}(2p_{34}-1)+T_Bp_{12}(1-2p_{34})\\&=(T_A-T_B)p_{12}(2p_{34}-1)
\end{align*}
$\because$ $T_A \geq T_B, p_{34} \geq \frac{1}{2}$, i.e. $2p_{34}-1\geq 0$, $\therefore$ $P_{\pi=(1,1,1,1)}^{(1)}(3,W,V) \geq P_{\pi=(0,1,1,1)}^{(1)}(3,W,V)$ for all $V=[v_1, v_2, v_3, v_4]$.

This implies that it is always sensible for $a_1$ to try to win if the performance vector $W$ of first two week is $[2,2,0,0]$.
\end{proof}

\subsection{Proof of Theorem \ref{1201t}}
\begin{proof}
Under $W=[1,2,0,1]$ or $[0,1,1,2]$, we have
\begin{align*}
P_{\pi=(1,1,1,1)}^{(1)}(3,W,V)-P_{\pi=(0,1,1,1)}^{(1)}(3,W,V)&=(p_{21}-1)p_{34}T_A+(p_{21}-1)p_{43}\frac{1}{3}(T_A+T_B+T_C)\\&\ \ \ \ +p_{12}p_{43}\frac{1}{3}(T_A+T_B+T_C)+ p_{21}p_{43}\frac{1}{2}(T_A+T_B)\\& = p_{12}\left(-\frac{2}{3}T_A+\frac{1}{3}T_B+\frac{1}{3}T_C+p_{34}(\frac{5}{6}T_A-\frac{1}{6}T_B-\frac{2}{3}T_C)\right)
\end{align*}

Since we want to find $V=(v_1,v_2,v_3,v_4)$ such that $P_{\pi=(1,1,1,1)}^{(1)}(3,W,V)-P_{\pi=(0,1,1,1)}^{(1)}(3,W,V) \geq 0$, which is equal to solve for $p_{34} \geq \frac{4T_A-2T_B-2T_C}{5T_A-T_B-4T_C}$. Hence, we want to solve the following inequality:

\begin{equation}
\frac{v_3}{v_3+v_4} \geq \frac{4T_A-2T_B-2T_C}{5T_A-T_B-4T_C}
\end{equation}

By expanding the right side of the inequality, we can get 
\begin{align*}
& 3v_2v_3v_4^2+2v_2v_4^3+6v_2^2v_3^2v_4+v_2v_3^2v_4+v_3^3v_4+2v_3v_4^3+10v_3^2v_4^2-3v_2^2v_3^3 \\ & -
v_2^2v_3v_4^2-2v_2^2v_4^3-v_2v_3^4-3v_2v_3^3v_4-2v_2v_3^2v_4^2-v_3^4v_4-6v_3^3v_4^2-6v_3^2v_4^3 \leq 0
\end{align*}
which leads to theorem \ref{1201t}.
\end{proof}

\subsection{Proof of Theorem \ref{1210t}}
\begin{proof}
Under $W=[1,2,1,0]$, we have
\begin{align*}
P_{\pi=(1,1,1,1)}^{(1)}(3,W,V)-P_{\pi=(0,1,1,1)}^{(1)}(3,W,V)&=(p_{21}-1)p_{34}T_B+(p_{21}-1)p_{43}\frac{1}{3}(T_A+T_B+T_C)\\&\ \ \ \ +p_{12}p_{34}\frac{1}{2}(T_A+T_B)+ p_{12}p_{43}\frac{1}{3}(T_A+T_B+T_C)\\& = p_{12}\left(\frac{1}{6}T_A+\frac{1}{6}T_B-\frac{1}{3}T_C+p_{34}(\frac{1}{6}T_A-\frac{5}{6}T_B+\frac{2}{3}T_C)\right)
\end{align*}

Since we want to find $V=(v_1,v_2,v_3,v_4)$ such that $P_{\pi=(1,1,1,1)}^{(1)}(3,W,V)-P_{\pi=(0,1,1,1)}^{(1)}(3,W,V) \geq 0$, which is equal to solve for $p_{34} \geq \frac{T_A+T_B-2T_C}{-T_A+5T_B-4T_C}$. Hence, we want to solve the following inequality:

\begin{equation}
\frac{v_3}{v_3+v_4} \geq \frac{T_A+T_B-2T_C}{-T_A+5T_B-4T_C}
\end{equation}

By expanding the right side of the inequality, we can get 
\begin{align*}
& 3v_2^2v_3^2v_4+v_2v_3^4+v_2v_3^3v_4+v_3^3v_4^2+4v_3^3v_4+4v_3^2v_4^2+5v_2v_3v_4 \\ & -4v_2^2v_3v_4^2-5v_2^2v_4^3-4v_3^4v_4-2v_2v_3v_4^2-v_3^2v_4-v_3v_4^2 \leq 0
\end{align*}
which leads to theorem \ref{1210t}.
\end{proof}

\subsection{Proof of Lemma \ref{T1}}
\begin{proof}
By theorem \ref{thm:ordertourn}, we know that if $v_1 \geq v_2 \geq v_3 \geq v_4$, then
\begin{equation}
T_{1A} \geq T_{1B} \geq T_{1C}
\end{equation}

Similar to theorem \ref{thm:ordertourn}, we calculate $T_{2A}$, $T_{2B}$, $T_{2C}$,
\begin{align*}
 T_{2A} &= \frac{v_2}{v_2+v_3} \left(\frac{v_2}{v_1+v_2}\frac{v_1}{v_1+v_4}+\frac{v_2}{v_2+v_4}\frac{v_4}{v_1+v_4} \right)=\frac{v_2^2 [v_1 (v_2+v_4 )+v_4 (v_1+v_2 )]}{(v_1+v_2 )(v_1+v_4 )(v_2+v_3 )(v_2+v_4 )} \\
 T_{2B} &= \frac{v_2}{v_2+v_4} \left(\frac{v_2}{v_1+v_2}\frac{v_1}{v_1+v_3}+\frac{v_2}{v_2+v_3}\frac{v_3}{v_1+v_3} \right)=\frac{v_2^2 [v_1 (v_2+v_3 )+v_3 (v_1+v_2 )]}{(v_1+v_2 )(v_1+v_3 )(v_2+v_3 )(v_2+v_4 )} \\
 T_{2C} &= \frac{v_2}{v_1+v_2} \left(\frac{v_2}{v_2+v_3}\frac{v_3}{v_3+v_4}+\frac{v_2}{v_2+v_4}\frac{v_4}{v_3+v_4} \right)=\frac{v_2^2 [v_3 (v_2+v_4 )+v_4 (v_2+v_3 )]}{(v_1+v_2 )(v_3+v_4 )(v_2+v_3 )(v_2+v_4 )}
\end{align*}

Compare $T_{2B}$ and $T_{2A}$, we compute 
\begin{align*}  
T_{2B}-T_{2A} &= \left(\frac{v_2^2}{(v_1+v_2 )(v_2+v_3 )(v_2+v_4)} \right)\left(\frac{v_1 (v_2+v_3 )+v_3 (v_1+v_2 )}{v_1+v_3}-\frac{v_1 (v_2+v_4 )+v_4 (v_1+v_2 )}{v_1+v_4} \right)\\
&=\left(\frac{v_2^2}{(v_1+v_2 )(v_2+v_3 )(v_2+v_4)} \right)(v_1 v_2+2v_1 v_3+v_2 v_3 )(v_1+v_4 )-(v_1 v_2+2v_1 v_4+v_2 v_4 )(v_1+v_3 )\\&=\left(\frac{v_1^2}{(v_1+v_4 )(v_1+v_2 )(v_1+v_3)} \right)(2v_1^2 v_3-2v_1^2 v_4)=\left(\frac{v_1^2}{(v_1+v_4 )(v_1+v_2 )(v_1+v_3)} \right)(2v_1^2 (v_3-v_4))
\end{align*}
$\because$ $v_1 \geq v_2 \geq v_3 \geq v_4$, $\therefore$ $T_{2B}-T_{2A}\geq 0$ 

Then, compute $T_{2A}-T_{2C}$, we can get
\begin{align*}  
T_{2A}-T_{2C} &= \left(\frac{v_2^2}{(v_1+v_2 )(v_2+v_3 )(v_2+v_4)} \right)\left(\frac{v_1 (v_2+v_4 )+v_4 (v_1+v_2 )}{v_1+v_4}-\frac{v_3 (v_2+v_4 )+v_4 (v_2+v_3 )}{v_3+v_4} \right)\\
&=\left(\frac{v_2^2}{(v_1+v_2 )(v_2+v_3 )(v_2+v_4)} \right)(v_1 v_2+2v_1 v_4+v_2 v_4 )(v_3+v_4 )-(v_2 v_3+v_2 v_4+2v_3 v_4 )(v_1+v_4 )\\&=\left(\frac{v_1^2}{(v_1+v_4 )(v_1+v_2 )(v_1+v_3)} \right)(2v_1 v_4^2-2v_3 v_4^2)=\left(\frac{v_1^2}{(v_1+v_4 )(v_1+v_2 )(v_1+v_3)} \right)(2v_4^2 (v_1-v_3))
\end{align*}
$\because$ $v_1 \geq v_2 \geq v_3 \geq v_4$, $\therefore$ $T_{2A}-T_{2C}\geq 0$ 

Hence, we can draw the conclusion that 
\begin{equation}
T_{2B} \geq T_{2A} \geq T_{2C}
\end{equation}

Next, we calculate $T_{3A}$, $T_{3B}$, $T_{3C}$,
\begin{align*}
 T_{3A} &= \frac{v_3}{v_2+v_3} \left(\frac{v_3}{v_1+v_3}\frac{v_1}{v_1+v_4}+\frac{v_3}{v_3+v_4}\frac{v_4}{v_1+v_4} \right)=\frac{v_3^2 [v_1 (v_3+v_4 )+v_4 (v_1+v_3 )]}{(v_1+v_3 )(v_1+v_4 )(v_2+v_3 )(v_3+v_4 )} \\
 T_{3B} &= \frac{v_3}{v_1+v_3} \left(\frac{v_3}{v_2+v_3}\frac{v_2}{v_2+v_4}+\frac{v_3}{v_3+v_4}\frac{v_4}{v_2+v_4} \right)=\frac{v_3^2 [v_2 (v_3+v_4 )+v_4 (v_2+v_3 )]}{(v_1+v_3 )(v_2+v_3 )(v_2+v_4 )(v_3+v_4 )} \\
 T_{3C} &= \frac{v_3}{v_3+v_4} \left(\frac{v_3}{v_1+v_3}\frac{v_1}{v_1+v_2}+\frac{v_3}{v_2+v_3}\frac{v_2}{v_1+v_2} \right)=\frac{v_3^2 [v_1 (v_2+v_3 )+v_2 (v_1+v_3 )]}{(v_1+v_2 )(v_1+v_3 )(v_2+v_3 )(v_3+v_4 )}
\end{align*}

Compare $T_{3C}$ and $T_{3A}$, we compute 
\begin{align*}  
T_{3C}-T_{3A} &= \left(\frac{v_3^2}{(v_1+v_3 )(v_2+v_3 )(v_3+v_4)} \right)\left(\frac{v_1 (v_2+v_3 )+v_2 (v_1+v_3 )}{v_1+v_2}-\frac{v_1 (v_3+v_4 )+v_4 (v_1+v_3 )}{v_1+v_4} \right)\\
&=\left(\frac{v_3^2}{(v_1+v_3 )(v_2+v_3 )(v_3+v_4)} \right)(2v_1 v_2+v_1 v_3+v_2 v_3 )(v_1+v_4 )-(v_1 v_3+2v_1 v_4+v_3 v_4 )(v_1+v_2 )\\&=\left(\frac{v_1^2}{(v_1+v_4 )(v_1+v_2 )(v_1+v_3)} \right)(2v_1^2 v_2-2v_1^2 v_4)=\left(\frac{v_3^2}{(v_1+v_3 )(v_2+v_3 )(v_3+v_4)} \right)(2v_1^2 (v_2-v_4))
\end{align*}
$\because$ $v_1 \geq v_2 \geq v_3 \geq v_4$, $\therefore$ $T_{3C}-T_{3A}\geq 0$ 

Then, we compare $T_{3A}$ and $T_{3B}$,
\begin{align*}  
T_{3A}-T_{3B} &= \left(\frac{v_3^2}{(v_1+v_3 )(v_2+v_3 )(v_3+v_4)} \right)\left(\frac{v_1 (v_3+v_4 )+v_4 (v_1+v_3 )}{v_1+v_4}-\frac{v_2 (v_3+v_4 )+v_4 (v_2+v_3 )}{v_2+v_4} \right)\\
&=\left(\frac{v_3^2}{(v_1+v_3 )(v_2+v_3 )(v_3+v_4)} \right)((v_1 v_3+2v_1 v_4+v_3 v_4 )(v_2+v_4 )-(v_2 v_3+2v_2 v_4+v_3 v_4 )(v_1+v_4 )\\&=\left(\frac{v_1^2}{(v_1+v_4 )(v_1+v_2 )(v_1+v_3)} \right)(2v_1 v_4^2-2v_2 v_4^2)=\left(\frac{v_3^2}{(v_1+v_3 )(v_2+v_3 )(v_3+v_4)} \right)(2v_4^2 (v_1-v_2))
\end{align*}
$\because$ $v_1 \geq v_2 \geq v_3 \geq v_4$, $\therefore$ $T_{3A}-T_{3B}\geq 0$ 

Hence, we can draw the conclusion that 
\begin{equation}
T_{3C} \geq T_{3A} \geq T_{3B}
\end{equation}

Finally, we calculate $T_{4A}$, $T_{4B}$, $T_{4C}$,
\begin{align*}
 T_{4A} &= \frac{v_4}{v_1+v_4} \left(\frac{v_4}{v_2+v_4}\frac{v_2}{v_2+v_3}+\frac{v_4}{v_3+v_4}\frac{v_3}{v_2+v_3} \right)=\frac{v_4^2 [v_2 (v_3+v_4 )+v_3 (v_2+v_4 )]}{(v_1+v_4 )(v_2+v_4 )(v_2+v_3 )(v_3+v_4 )} \\
 T_{4B} &= \frac{v_4}{v_2+v_4} \left(\frac{v_4}{v_1+v_4}\frac{v_1}{v_1+v_3}+\frac{v_4}{v_3+v_4}\frac{v_3}{v_1+v_3} \right)=\frac{v_4^2 [v_1 (v_3+v_4 )+v_3 (v_1+v_4 )]}{(v_1+v_4 )(v_1+v_3 )(v_2+v_4 )(v_3+v_4 )} \\
 T_{4C} &= \frac{v_4}{v_3+v_4} \left(\frac{v_4}{v_1+v_4}\frac{v_1}{v_1+v_2}+\frac{v_4}{v_2+v_4}\frac{v_2}{v_1+v_2} \right)=\frac{v_4^2 [v_1 (v_2+v_4 )+v_2 (v_1+v_4 )]}{(v_1+v_2 )(v_1+v_4 )(v_2+v_4 )(v_3+v_4 )}
\end{align*}

Compare $T_{4C}$ and $T_{4B}$, we compute 
\begin{align*}  
T_{4C}-T_{4B} &= \left(\frac{v_4^2}{(v_1+v_4 )(v_2+v_4 )(v_3+v_4)} \right)\left(\frac{v_1 (v_2+v_4 )+v_2 (v_1+v_4 )}{v_1+v_2}-\frac{v_1 (v_3+v_4 )+v_3 (v_1+v_4 )}{v_1+v_3} \right)\\
&=\left(\frac{v_4^2}{(v_1+v_4 )(v_2+v_4 )(v_3+v_4)} \right)(2v_1 v_2+v_1 v_4+v_2 v_4 )(v_1+v_3 )-(2v_1 v_3+v_1 v_4+v_3 v_4 )(v_1+v_2 )\\&=\left(\frac{v_4^2}{(v_1+v_4 )(v_2+v_4 )(v_3+v_4)} \right)(2v_1^2 v_2-2v_1^2 v_3)=\left(\frac{v_4^2}{(v_1+v_4 )(v_2+v_4 )(v_3+v_4)} \right)(2v_1^2 (v_2-v_3))
\end{align*}
$\because$ $v_1 \geq v_2 \geq v_3 \geq v_4$, $\therefore$ $T_{4C}-T_{4B}\geq 0$ 

Then, we compare $T_{4B}$ and $T_{4A}$,
\begin{align*}  
T_{4B}-T_{4A} &= \left(\frac{v_4^2}{(v_1+v_4 )(v_2+v_4 )(v_3+v_4)} \right)\left(\frac{v_1 (v_3+v_4 )+v_3 (v_1+v_4 )}{v_1+v_3}-\frac{v_2 (v_3+v_4 )+v_3 (v_2+v_4 )}{v_2+v_3} \right)\\
&=\left(\frac{v_4^2}{(v_1+v_4 )(v_2+v_4 )(v_3+v_4)} \right)(2v_1 v_3+v_1 v_4+v_3 v_4 )(v_2+v_3 )-(2v_2 v_3+v_2 v_4+v_3 v_4 )(v_1+v_3 )\\&=\left(\frac{v_4^2}{(v_1+v_4 )(v_2+v_4 )(v_3+v_4)} \right)(2v_1 v_3^2-2v_2 v_3^2)=\left(\frac{v_4^2}{(v_1+v_4 )(v_2+v_4 )(v_3+v_4)} \right)(2v_3^2 (v_1-v_2))
\end{align*}
$\because$ $v_1 \geq v_2 \geq v_3 \geq v_4$, $\therefore$ $T_{4B}-T_{4A}\geq 0$ 

Hence, we can draw the conclusion that 
\begin{equation}
T_{4C} \geq T_{4B} \geq T_{4A}
\end{equation}
\end{proof}

\end{document}